\documentclass{amsart}

\usepackage[a4paper,margin=1.1in]{geometry}
\usepackage{amsmath,amssymb,amsthm,mathtools}
\usepackage{microtype}
\usepackage{hyperref}

\hypersetup{hidelinks}

\newtheorem{theorem}{Theorem}[section]
\newtheorem{proposition}[theorem]{Proposition}
\newtheorem{lemma}[theorem]{Lemma}
\newtheorem{corollary}[theorem]{Corollary}

\title{Five-Term and Higher Congruences Involving Arbitrary Sets\\
and Short Intervals Modulo a Prime}
\author{Yao Zhi}
\email{ericyao2026@gmail.com}
\date{}

\begin{document}

\maketitle

\begin{abstract}
	We obtain asymptotic formulas for additive congruences
	\[
	\sum_{i=1}^r m_i x_i^{-s}\equiv \lambda \pmod p,
	\]
	where the \(m_i\) range over arbitrary subsets of \(\mathbb F_p^\ast\) and the \(x_i\) over shifted intervals. For five terms, in the balanced case of common cardinality \(N\), the asymptotic holds uniformly in \(\lambda\) whenever
	\[
	N>p^{14/29+\varepsilon},
	\]
	giving a genuine sub-square-root range. The main input is a centered fourth-moment estimate for the associated double exponential sums. The same method yields sub-square-root thresholds for every fixed \(r\ge 5\), including \(N>p^{8/17+\varepsilon}\) for six terms, with
	\[
	\alpha_r=\frac13+\frac{4}{9\sqrt r}+O(r^{-1})
	\]
	as \(r\to\infty\).
\end{abstract}

\medskip
\noindent\textit{2020 Mathematics Subject Classification.}
Primary 11T23; Secondary 11D79, 11L07.

\smallskip
\noindent\textit{Keywords.}
Finite fields, exponential sums, short intervals, additive energy,
congruences, arbitrary sets.

\section{Introduction}

Let \(p\) be a sufficiently large prime, let \(\mathbb F_p\) denote
the field of \(p\) elements, and write
\[
\mathrm e_p(z)=\exp(2\pi iz/p).
\]
Throughout the paper, \(s\geq1\) is a fixed integer. For \(r\geq1\),
let
\[
\mathcal M_i\subseteq\mathbb F_p^\ast,
\qquad
|\mathcal M_i|=M,
\]
and let
\[
\mathcal X_i
=
L_i+\{1,\ldots,H\}
\subseteq\mathbb F_p^\ast,
\qquad
1\leq i\leq r.
\]
We denote by \(T_r(\lambda)\) the number of solutions to
\[
m_1x_1^{-s}+\cdots+m_rx_r^{-s}
\equiv\lambda\pmod p,
\qquad
m_i\in\mathcal M_i,\quad x_i\in\mathcal X_i.
\]
The expected main term is \((HM)^r/p\), and the problem is to obtain
an asymptotic formula uniformly in \(\lambda\), in the arbitrary sets
\(\mathcal M_i\), and in the positions of the intervals
\(\mathcal X_i\).

The study of multiplicative congruences with variables restricted to
short ranges is closely connected with mean values of character sums.
A classical result of Ayyad, Cochrane and Zheng \cite{ACZ} gives an
asymptotic formula for
\[
x_1x_2\equiv x_3x_4\pmod p
\]
in boxes and, equivalently, a fourth-moment estimate for multiplicative
character sums over intervals. Garaev and Karatsuba
\cite{GaraevKaratsuba05,GaraevKaratsuba07} subsequently developed
methods for multiplicative representation problems. A broad account
of modular hyperbolas and related distribution questions is given by
Shparlinski \cite{ShparlinskiHyperbolas}.

Several later developments are particularly relevant to the present
setting. Bourgain, Garaev, Konyagin and Shparlinski \cite{BGKS}
obtained strong bounds for multiplicative congruences with variables
from short intervals, while Bourgain and Garaev \cite{BG} developed
additive-structure estimates for reciprocal intervals and applications
to multilinear Kloosterman sums. Shparlinski
\cite{ShparlinskiRatios} studied linear congruences with ratios of
interval variables. Problems involving a short interval and a
completely arbitrary subset of \(\mathbb F_p^\ast\) were investigated
by Shkredov and Shparlinski \cite{ShkredovShparlinski}, Banks and
Shparlinski \cite{BanksShparlinski}, and Bag and Shparlinski
\cite{BagShparlinski}. Additive energies between a set and its
multiplicative dilates have also played an important role in
finite-field sum-product theory. The basic total-energy identity goes
back to Bourgain, Katz and Tao \cite{BKT}; average and structural
refinements were developed, among others, by Glibichuk
\cite{GlibichukEnergy} and Murphy and Petridis
\cite{MurphyPetridis}. Related recent developments on energies and
mean values of rational or modular-root exponential sums can be found
in \cite{KerrEtAl,KohShparlinski,XuZhang}.

The work most directly connected with our problem is that of Garaev
and Shparlinski \cite{GaraevShparlinski}. For an arbitrary set
\(\mathcal M\subseteq\mathbb F_p^\ast\) and an arbitrary shifted
interval \(\mathcal X\), they obtained pointwise estimates for double
exponential sums of the form
\[
\sum_{m\in\mathcal M}\sum_{x\in\mathcal X}
\mathrm e_p(amx^{-s})
\]
and used them, together with multiplicative collision estimates, to
study additive congruences involving the quantities \(mx^{-s}\).
Their six-term theorem gives, in the balanced case, a uniform
asymptotic formula when
\[
H=M>p^{17/35+\varepsilon}.
\]
They explicitly posed the problem of obtaining an asymptotic formula
for the five-term analogue when
\[
H=M<p^{1/2},
\]
noting that the problem was open even when all five numerator sets
coincide and all five interval shifts vanish; see
\cite[Section~1.2]{GaraevShparlinski}.

Garaev, Pardo and Shparlinski
\cite{GaraevPardoShparlinski} subsequently obtained a ternary
representation theorem at the exponent \(14/29\). More precisely,
for an arbitrary set \(\mathcal M\subseteq\mathbb F_p^\ast\) with
\[
|\mathcal M|=\lfloor p^{14/29}\rfloor
\]
and the initial interval of length
\[
H=\lfloor p^{14/29+\varepsilon}\rfloor,
\]
every residue class has a representation as a sum of three terms
\(m/x^s\). Their conclusion is a representability result rather than
an asymptotic formula for the number of representations. A recent preprint of Bag and Mazumder
\cite{BagMazumder} studies higher moments of related double
exponential sums. The fourth moment used below is different in that
the averaging variable is the additive frequency, which is precisely
the variable arising from the Fourier expansion of \(T_r(\lambda)\).

Our main observation is that the nonzero-frequency fourth moment
should be centered before the known information on additive energies
of multiplicative dilates is used. This leads to a substantially
sharper frequency-side fourth moment in the sub-square-root range.
For
\[
\mathcal M\subseteq\mathbb F_p^\ast,
\qquad
|\mathcal M|=M,
\]
and a shifted interval
\[
\mathcal X=L+\{1,\ldots,H\}\subseteq\mathbb F_p^\ast,
\]
define
\[
S_{\mathcal M,\mathcal X}(a)
=
\sum_{m\in\mathcal M}\sum_{x\in\mathcal X}
\mathrm e_p(amx^{-s})
\]
and put
\[
\mu=\min\{M,p^{1/2}\}.
\]

\begin{theorem}\label{thm:fourth-moment}
For every fixed \(s\geq1\),
\[
\sum_{a\in\mathbb F_p^\ast}
|S_{\mathcal M,\mathcal X}(a)|^4
\ll_s
pH^2M^2(H+\mu)^2p^{o(1)}.
\]
The estimate is uniform in \(\mathcal M\), in the position of
\(\mathcal X\), and in \(H\) and \(M\).
\end{theorem}

The main content of Theorem~\ref{thm:fourth-moment} lies in the range
\(H\leq p^{1/2}\). The extension to longer intervals, which we do not
attempt to optimize, follows by subdividing the interval and is
included only to avoid a separate large-\(N\) argument later.

In particular, if \(H,M\leq p^{1/2}\), then
\[
\sum_{a\in\mathbb F_p^\ast}
|S_{\mathcal M,\mathcal X}(a)|^4
\ll_s
pH^2M^2(H+M)^2p^{o(1)},
\]
and in the balanced case \(H=M=N\leq p^{1/2}\),
\[
\sum_{a\in\mathbb F_p^\ast}
|S_{\mathcal M,\mathcal X}(a)|^4
\ll_s
pN^6p^{o(1)}.
\]
This is the scale needed to pass below the square-root barrier in the
five-term problem.

The point of Theorem~\ref{thm:fourth-moment} is the centering. If
\[
r_{\mathcal M,\mathcal X}(t)
=
\#\{(m,x)\in\mathcal M\times\mathcal X:mx^{-s}=t\},
\]
then removal of the zero additive frequency leads to the exact
variance of the additive correlation of
\(r_{\mathcal M,\mathcal X}\). Decomposing this centered correlation
according to the ratio \((x/y)^s\) reduces the problem to an
interaction between the ratio multiplicity of the interval and the
centered additive energy
\[
E^+(\mathcal M,\lambda\mathcal M)-\frac{M^4}{p}.
\]
The dilate-energy estimates used at this point are classical; see
\cite{BKT,MurphyPetridis}. The gain comes from coupling them, on each
dyadic level, with the second moment of the interval-ratio
multiplicity instead of applying Cauchy--Schwarz globally.

Our first arithmetic consequence answers the five-term question in a
power range strictly below \(p^{1/2}\).

\begin{theorem}\label{thm:five-main}
Fix \(s\geq1\) and \(\varepsilon>0\). For \(1\leq i\leq5\), let
\[
\mathcal M_i\subseteq\mathbb F_p^\ast,
\qquad
|\mathcal M_i|=M,
\]
and let
\[
\mathcal X_i
=
L_i+\{1,\ldots,H\}
\subseteq\mathbb F_p^\ast.
\]
Assume that \(H\leq p^{1/2}\) and
\[
H^{27}M^{26}
>
p^{14+\varepsilon}
\left(H+\min\{M,p^{1/2}\}\right)^{24}.
\]
Then there exists \(\delta=\delta(s,\varepsilon)>0\) such that,
uniformly for \(\lambda\in\mathbb F_p\),
\[
T_5(\lambda)
=
\frac{H^5M^5}{p}
\left(1+O(p^{-\delta})\right).
\]
\end{theorem}

\begin{corollary}\label{cor:five-balanced}
Fix \(s\geq1\) and \(\varepsilon>0\). If
\[
H=M=N>p^{14/29+\varepsilon},
\]
then, for some \(\delta=\delta(s,\varepsilon)>0\),
\[
T_5(\lambda)
=
\frac{N^{10}}{p}
\left(1+O(p^{-\delta})\right)
\]
uniformly for \(\lambda\in\mathbb F_p\).
\end{corollary}

Since
\[
\frac{14}{29}=0.482758\ldots<\frac12,
\]
Corollary~\ref{cor:five-balanced} gives a five-term asymptotic formula
in a genuine sub-square-root range. To the best of our knowledge, this is the first asymptotic formula
for the five-term problem of \cite{GaraevShparlinski} in a power range
strictly below the square-root threshold. The sets
\(\mathcal M_1,\ldots,\mathcal M_5\) and the interval shifts
\(L_1,\ldots,L_5\) are allowed to be different.

The same fourth moment applies systematically to any fixed number
\(r\geq5\) of summands. For integers \(r\geq5\) and \(\ell\geq1\),
put
\[
\alpha_{r,\ell}
=
\frac{(\ell+1)(2\ell+r-4)}
{4\ell^2+(3r-8)\ell+r-4},
\]
and define
\[
\ell_r
=
\left\lceil
\frac{1+\sqrt{4r-7}}2
\right\rceil,
\qquad
\alpha_r=\alpha_{r,\ell_r}.
\]
When
\[
r=\ell_r^2-\ell_r+2,
\]
the parameters \(\ell_r\) and \(\ell_r+1\) give the same value of
\(\alpha_{r,\ell}\).

\begin{theorem}\label{thm:r-balanced}
Fix \(r\geq5\), \(s\geq1\) and \(\varepsilon>0\). Suppose that
\[
|\mathcal M_i|=|\mathcal X_i|=N,
\qquad
1\leq i\leq r,
\]
where the \(\mathcal M_i\subseteq\mathbb F_p^\ast\) are arbitrary and
the \(\mathcal X_i\subseteq\mathbb F_p^\ast\) are shifted intervals.
If
\[
N>p^{\alpha_r+\varepsilon},
\]
then there exists \(\delta=\delta(s,r,\varepsilon)>0\) such that
\[
T_r(\lambda)
=
\frac{N^{2r}}{p}
\left(1+O(p^{-\delta})\right)
\]
uniformly for \(\lambda\in\mathbb F_p\). Moreover,
\[
\alpha_r
=
\frac13+\frac{4}{9\sqrt r}+O(r^{-1})
\]
as \(r\to\infty\).
\end{theorem}

For reference, the first few values are
\[
\begin{array}{c|c|c}
	r & \ell_{\mathrm{opt}} & \alpha_r \\ \hline
	5 & 3   & 14/29\\
	6 & 3   & 8/17\\
	7 & 3   & 6/13\\
	8 & 3,4 & 5/11\\
	9 & 4   & 13/29\\
	10& 4   & 35/79
\end{array}
\]
and, in particular,
\[
\alpha_6=\frac8{17}<\frac{17}{35}.
\]
Thus Theorem~\ref{thm:r-balanced} improves the balanced six-term
threshold of Garaev and Shparlinski
\cite[Theorem~1.3]{GaraevShparlinski}.

Moreover,
\[
\alpha_r
\leq
\alpha_{r,3}
=
\frac{2r+4}{5r+4}
<
\frac12
\qquad (r\geq5),
\]
so Theorem~\ref{thm:r-balanced} gives a genuine sub-square-root
asymptotic range for every fixed \(r\geq5\).

The proof has two principal analytic inputs. The first is
Theorem~\ref{thm:fourth-moment}; the second is the pointwise estimate
of Garaev and Shparlinski recorded in
Lemma~\ref{lem:pointwise}. Section~\ref{sec:auxiliary} collects the
auxiliary estimates. The centered fourth moment is proved in
Section~\ref{sec:fourth}. In Section~\ref{sec:congruences} we derive
a general Fourier error estimate and prove
Theorem~\ref{thm:five-main} and Corollary~\ref{cor:five-balanced}.
The balanced \(r\)-term threshold is optimized in
Section~\ref{sec:optimization}.

All implied constants below may depend on the fixed parameters
\(s,r,\ell\), and on \(\varepsilon\) when it occurs, but not on \(p\),
the sets, the interval shifts, \(\lambda\), or the nonzero frequencies.
Every occurrence of \(p^{o(1)}\) is uniform in these varying objects:
for each fixed \(\eta>0\), it may be replaced by \(O(p^\eta)\), with
an implied constant depending only on \(\eta\) and the fixed
parameters.

\paragraph{\textbf{The role of AI tools}}

During the preparation of this work, the author used ChatGPT 5.5 (Plus) as an auxiliary research tool. In particular, reference \cite{ACZ} was included following a suggestion made by the AI tool, and parts of the proof strategy for Theorem 1.1 were developed with partial inspiration from AI-assisted discussions. AI tools were also used to assist with some computations, algebraic manipulations, and verification steps in other parts of the manuscript. All mathematical arguments, references, computations, and conclusions appearing in the final manuscript were independently checked and verified by the author, who takes full responsibility for the content of the paper.

\section{Auxiliary estimates}\label{sec:auxiliary}

Throughout this section,
\[
\mathcal X=L+\{1,\ldots,H\}\subseteq\mathbb F_p^\ast,
\qquad
\mathcal M\subseteq\mathbb F_p^\ast,
\qquad
|\mathcal M|=M.
\]
For \(\lambda\in\mathbb F_p^\ast\), define
\[
q_s(\lambda)
=
\#\left\{
(x,y)\in\mathcal X^2:(x/y)^s=\lambda
\right\}.
\]

\begin{lemma}\label{lem:ratio-moment}
Let \(s\geq1\) be fixed. If \(H\leq p^{1/2}\), then
\[
\sum_{\lambda\in\mathbb F_p^\ast}q_s(\lambda)^2
\ll_s H^2p^{o(1)}.
\]
\end{lemma}

\begin{proof}
Let \(\widehat{\mathbb F_p^\ast}\) denote the multiplicative character
group of \(\mathbb F_p^\ast\). By multiplicative-character
orthogonality,
\[
\sum_{\lambda\in\mathbb F_p^\ast}q_s(\lambda)^2
=
\frac{1}{p-1}
\sum_{\chi\in\widehat{\mathbb F_p^\ast}}
\left|
\sum_{x\in\mathcal X}\chi^s(x)
\right|^4.
\]
The homomorphism \(\chi\mapsto\chi^s\) has fibres of cardinality at
most \(s\). Hence
\[
\sum_{\chi\in\widehat{\mathbb F_p^\ast}}
\left|
\sum_{x\in\mathcal X}\chi^s(x)
\right|^4
\leq
s
\sum_{\psi\in\widehat{\mathbb F_p^\ast}}
\left|
\sum_{x\in\mathcal X}\psi(x)
\right|^4.
\]
Let \(\psi_0\) denote the principal multiplicative character.
The fourth-moment estimate of Ayyad, Cochrane and Zheng
\cite{ACZ}, uniformly for intervals of arbitrary position, gives
\[
\sum_{\substack{\psi\in\widehat{\mathbb F_p^\ast}\\
\psi\neq\psi_0}}
\left|
\sum_{x\in\mathcal X}\psi(x)
\right|^4
\ll pH^2(\log p)^2.
\]
The principal character contributes \(H^4\), which is \(O(pH^2)\)
when \(H\leq p^{1/2}\). The result follows.
\end{proof}

We also use the following pointwise estimate of Garaev and
Shparlinski \cite[Theorem~1.2]{GaraevShparlinski}.

\begin{lemma}\label{lem:pointwise}
Let \(s,\ell\geq1\) be fixed. For every \(a\in\mathbb F_p^\ast\),
\[
\left|
\sum_{m\in\mathcal M}\sum_{x\in\mathcal X}
\mathrm e_p(amx^{-s})
\right|
\ll_{s,\ell}
HM
\left(
\frac{p}{MH^{2\ell/(\ell+1)}}+\frac1M
\right)^{1/(2\ell)}
p^{o(1)}.
\]
\end{lemma}

For \(\lambda\in\mathbb F_p^\ast\), put
\[
E_{\mathcal M}(\lambda)
=
E^+(\mathcal M,\lambda\mathcal M)
=
\#\left\{
(m_1,m_2,m_3,m_4)\in\mathcal M^4:
m_1-m_2=\lambda(m_3-m_4)
\right\}
\]
and
\[
D_{\mathcal M}(\lambda)
=
E_{\mathcal M}(\lambda)-\frac{M^4}{p}.
\]
Additive Fourier inversion shows that
\[
E_{\mathcal M}(\lambda)
=
\frac1p
\sum_{a\in\mathbb F_p}
\left|
\sum_{m\in\mathcal M}\mathrm e_p(am)
\right|^2
\left|
\sum_{m\in\mathcal M}\mathrm e_p(a\lambda m)
\right|^2,
\]
so in particular \(D_{\mathcal M}(\lambda)\geq0\).

The first estimate in the next lemma is the special case
\(A=B=\mathcal M\) of the elementary dilate-energy inequality in
\cite[Lemma~3]{MurphyPetridis}. The total-energy calculation behind
the second estimate goes back to Bourgain, Katz and Tao
\cite{BKT}; see also \cite{MurphyPetridis}. We include the short proof
because both forms are used below. More refined average estimates are
available under additional size hypotheses; see, for example,
\cite{GlibichukEnergy}.

\begin{lemma}\label{lem:dilate-energy}
For every \(\Lambda\subseteq\mathbb F_p^\ast\),
\[
\sum_{\lambda\in\Lambda}E_{\mathcal M}(\lambda)
\leq M^4+|\Lambda|M^2.
\]
Moreover,
\begin{equation}\label{eq:total-dilate-energy}
\sum_{\lambda\in\mathbb F_p^\ast}D_{\mathcal M}(\lambda)
=
pM^2\left(1-\frac Mp\right)^2.
\end{equation}
\end{lemma}

\begin{proof}
The first sum counts solutions of
\[
m_1-m_2=\lambda(m_3-m_4),
\qquad
\lambda\in\Lambda.
\]
If \(m_3=m_4\), then \(m_1=m_2\), giving
\(|\Lambda|M^2\) solutions. Otherwise the quadruple
\((m_1,m_2,m_3,m_4)\) determines at most one value of \(\lambda\).
This proves the first assertion.

For the second, summing \(E_{\mathcal M}(\lambda)\) over
\(\lambda\neq0\), the solutions with both differences equal to zero
contribute \((p-1)M^2\). If both differences are nonzero, the
quadruple determines a unique \(\lambda\in\mathbb F_p^\ast\), and
there are \((M^2-M)^2\) such quadruples. Thus
\[
\sum_{\lambda\neq0}E_{\mathcal M}(\lambda)
=
(p-1)M^2+(M^2-M)^2.
\]
Subtracting \((p-1)M^4/p\) gives
\[
pM^2-2M^3+\frac{M^4}{p}
=
pM^2\left(1-\frac Mp\right)^2,
\]
as required.
\end{proof}

\section{A centered fourth-moment estimate}\label{sec:fourth}

We now prove Theorem~\ref{thm:fourth-moment}. The main argument is
first carried out when \(H\leq p^{1/2}\).

Define
\[
r(t)
=
\#\left\{
(m,x)\in\mathcal M\times\mathcal X:
mx^{-s}=t
\right\}.
\]
Then
\[
S_{\mathcal M,\mathcal X}(a)
=
\sum_{t\in\mathbb F_p}r(t)\mathrm e_p(at).
\]
Write
\[
(r\circ r)(z)
=
\sum_{t\in\mathbb F_p}r(t)r(t-z).
\]
By additive Fourier orthogonality,
\[
\frac1p
\sum_{a\in\mathbb F_p}
|S_{\mathcal M,\mathcal X}(a)|^4
=
\sum_{z\in\mathbb F_p}(r\circ r)(z)^2.
\]
Since \(S_{\mathcal M,\mathcal X}(0)=HM\) and
\(\sum_z(r\circ r)(z)=H^2M^2\), we obtain
\begin{equation}\label{eq:centered-fourth}
\frac1p
\sum_{a\in\mathbb F_p^\ast}
|S_{\mathcal M,\mathcal X}(a)|^4
=
\sum_{z\in\mathbb F_p}
\left(
(r\circ r)(z)-\frac{H^2M^2}{p}
\right)^2.
\end{equation}

For \(x,y\in\mathcal X\), define
\[
\nu_{x,y}(z)
=
\#\left\{
(m,n)\in\mathcal M^2:
mx^{-s}-ny^{-s}=z
\right\}
\]
and
\[
F_{x,y}(z)
=
\nu_{x,y}(z)-\frac{M^2}{p}.
\]
Since
\[
r\circ r=\sum_{x,y\in\mathcal X}\nu_{x,y},
\]
equation \eqref{eq:centered-fourth} and Minkowski's inequality give
\begin{equation}\label{eq:minkowski-reduction}
\left(
\frac1p
\sum_{a\in\mathbb F_p^\ast}
|S_{\mathcal M,\mathcal X}(a)|^4
\right)^{1/2}
\leq
\sum_{x,y\in\mathcal X}\|F_{x,y}\|_2.
\end{equation}

The equality
\[
m_1x^{-s}-n_1y^{-s}
=
m_2x^{-s}-n_2y^{-s}
\]
is equivalent to
\[
m_1-m_2
=
(x/y)^s(n_1-n_2).
\]
Consequently,
\[
\sum_{z\in\mathbb F_p}\nu_{x,y}(z)^2
=
E_{\mathcal M}\left((x/y)^s\right),
\]
and, since \(\sum_z\nu_{x,y}(z)=M^2\),
\[
\|F_{x,y}\|_2^2
=
D_{\mathcal M}\left((x/y)^s\right).
\]
Grouping the pairs \((x,y)\) according to \((x/y)^s\), we deduce
from \eqref{eq:minkowski-reduction} that
\begin{equation}\label{eq:ratio-energy-reduction}
\left(
\frac1p
\sum_{a\in\mathbb F_p^\ast}
|S_{\mathcal M,\mathcal X}(a)|^4
\right)^{1/2}
\leq
\sum_{\lambda\in\mathbb F_p^\ast}
q_s(\lambda)D_{\mathcal M}(\lambda)^{1/2}.
\end{equation}

Assume now that \(H\leq p^{1/2}\), and put
\[
\mu=\min\{M,p^{1/2}\}.
\]
For dyadic \(R\geq1\), let
\[
\Lambda_R
=
\left\{
\lambda\in\mathbb F_p^\ast:
R\leq q_s(\lambda)<2R
\right\}.
\]
Lemma~\ref{lem:ratio-moment} gives
\[
|\Lambda_R|
\ll_s
\frac{H^2}{R^2}p^{o(1)}.
\]
We claim that
\[
\sum_{\lambda\in\Lambda_R}
D_{\mathcal M}(\lambda)
\leq
M^2\left(\mu^2+|\Lambda_R|\right).
\]
Indeed, if \(M\leq p^{1/2}\), the first assertion of
Lemma~\ref{lem:dilate-energy} yields
\[
\sum_{\lambda\in\Lambda_R}D_{\mathcal M}(\lambda)
\leq
M^4+|\Lambda_R|M^2.
\]
If \(M>p^{1/2}\), then nonnegativity and
\eqref{eq:total-dilate-energy} give
\[
\sum_{\lambda\in\Lambda_R}D_{\mathcal M}(\lambda)
\leq pM^2=M^2\mu^2.
\]

By Cauchy--Schwarz on each dyadic block,
\[
\begin{aligned}
\sum_{\lambda\in\Lambda_R}
q_s(\lambda)D_{\mathcal M}(\lambda)^{1/2}
&\ll
R|\Lambda_R|^{1/2}
\left(
\sum_{\lambda\in\Lambda_R}
D_{\mathcal M}(\lambda)
\right)^{1/2} \\
&\ll_s
HM\mu\,p^{o(1)}
+
\frac{H^2M}{R}p^{o(1)}.
\end{aligned}
\]
For each fixed \(y\in\mathcal X\), the equation
\(x^s=\lambda y^s\) has at most \(\gcd(s,p-1)\leq s\) solutions in
\(\mathbb F_p^\ast\). Hence \(q_s(\lambda)\leq sH\), so there are only
\(O_s(\log p)\) nonempty dyadic levels. Summing over them gives
\begin{equation}\label{eq:ratio-energy-bound}
\sum_{\lambda\in\mathbb F_p^\ast}
q_s(\lambda)D_{\mathcal M}(\lambda)^{1/2}
\ll_s
HM(H+\mu)p^{o(1)}.
\end{equation}
Combining \eqref{eq:ratio-energy-reduction} and
\eqref{eq:ratio-energy-bound}, we obtain
\begin{equation}\label{eq:fourth-short}
\sum_{a\in\mathbb F_p^\ast}
|S_{\mathcal M,\mathcal X}(a)|^4
\ll_s
pH^2M^2(H+\mu)^2p^{o(1)}
\end{equation}
for \(H\leq p^{1/2}\).

It remains to treat \(H>p^{1/2}\). Partition \(\mathcal X\) into
disjoint consecutive subintervals
\[
\mathcal X
=
\mathcal X_1\sqcup\cdots\sqcup\mathcal X_K,
\qquad
|\mathcal X_j|\leq p^{1/2},
\qquad
K\ll Hp^{-1/2}.
\]
Writing
\[
S_{\mathcal M,\mathcal X}
=
\sum_{j=1}^K S_{\mathcal M,\mathcal X_j},
\]
we apply \eqref{eq:fourth-short} to each \(\mathcal X_j\). Since
\(|\mathcal X_j|\leq p^{1/2}\) and \(\mu\leq p^{1/2}\),
\[
\|S_{\mathcal M,\mathcal X_j}\|_4
\ll_s
p^{3/4}M^{1/2}p^{o(1)}.
\]
Minkowski's inequality therefore gives
\[
\|S_{\mathcal M,\mathcal X}\|_4
\ll_s
Hp^{1/4}M^{1/2}p^{o(1)}.
\]
Hence
\[
\sum_{a\in\mathbb F_p^\ast}
|S_{\mathcal M,\mathcal X}(a)|^4
\ll_s
pH^4M^2p^{o(1)}.
\]
As \(H>p^{1/2}\geq\mu\), the last expression is
\[
\ll_s
pH^2M^2(H+\mu)^2p^{o(1)}.
\]
This completes the proof of Theorem~\ref{thm:fourth-moment}.

\section{Five-term and higher congruences}\label{sec:congruences}

For \(1\leq i\leq r\), put
\[
S_i(a)
=
\sum_{m\in\mathcal M_i}
\sum_{x\in\mathcal X_i}
\mathrm e_p(amx^{-s}).
\]
By additive-character orthogonality,
\[
T_r(\lambda)-\frac{(HM)^r}{p}
=
\frac1p
\sum_{a\in\mathbb F_p^\ast}
S_1(a)\cdots S_r(a)\mathrm e_p(-a\lambda).
\]

\begin{proposition}\label{prop:master-error}
Let \(r\geq5\) and \(\ell\geq1\) be fixed. With
\(\mu=\min\{M,p^{1/2}\}\), uniformly for
\(\lambda\in\mathbb F_p\),
\begin{equation}\label{eq:master-error}
\left|
T_r(\lambda)-\frac{(HM)^r}{p}
\right|
\ll_{s,r,\ell}
H^{r-2}M^{r-2}(H+\mu)^2
\left(
\frac{p}{MH^{2\ell/(\ell+1)}}+\frac1M
\right)^{(r-4)/(2\ell)}
p^{o(1)}.
\end{equation}
Consequently,
\begin{equation}\label{eq:relative-master}
\frac{
\left|T_r(\lambda)-(HM)^r/p\right|
}{
(HM)^r/p
}
\ll_{s,r,\ell}
\frac{p(H+\mu)^2}{H^2M^2}
\left(
\frac{p}{MH^{2\ell/(\ell+1)}}+\frac1M
\right)^{(r-4)/(2\ell)}
p^{o(1)}.
\end{equation}
\end{proposition}

\begin{proof}
Let
\[
W
=
\max_{\substack{1\leq i\leq r\\a\in\mathbb F_p^\ast}}
|S_i(a)|.
\]
H\"older's inequality and Theorem~\ref{thm:fourth-moment} give
\[
\left|
T_r(\lambda)-\frac{(HM)^r}{p}
\right|
\ll_{s,r}
W^{r-4}H^2M^2(H+\mu)^2p^{o(1)}.
\]
Applying Lemma~\ref{lem:pointwise} to \(W\) yields
\eqref{eq:master-error}; division by \((HM)^r/p\) gives
\eqref{eq:relative-master}.
\end{proof}

\begin{corollary}\label{cor:r-unbalanced}
Fix \(r\geq5\), \(\ell\geq1\), \(s\geq1\) and
\(\varepsilon>0\). Suppose that \(H\leq p^{1/2}\) and
\[
H^{(2\ell+r-2)/(\ell+1)}
M^{2+(r-4)/(2\ell)}
>
p^{1+(r-4)/(2\ell)+\varepsilon}
\left(H+\min\{M,p^{1/2}\}\right)^2.
\]
Then, for some
\(\delta=\delta(s,r,\ell,\varepsilon)>0\),
\[
T_r(\lambda)
=
\frac{(HM)^r}{p}
\left(1+O(p^{-\delta})\right)
\]
uniformly for \(\lambda\in\mathbb F_p\).
\end{corollary}

\begin{proof}
Put \(q=(r-4)/(2\ell)\). Since \(H\leq p^{1/2}\),
\[
\frac{p}{MH^{2\ell/(\ell+1)}}\geq\frac1M.
\]
Thus \eqref{eq:relative-master} gives
\[
\frac{
\left|T_r(\lambda)-(HM)^r/p\right|
}{
(HM)^r/p
}
\ll_{s,r,\ell}
\frac{
p^{1+q}(H+\mu)^2
}{
H^{(2\ell+r-2)/(\ell+1)}
M^{2+q}
}
p^{o(1)}.
\]
The hypothesis gives a fixed power saving.
\end{proof}

\begin{proof}[Proof of Theorem~\ref{thm:five-main}]
Take \(r=5\) and \(\ell=3\) in
Proposition~\ref{prop:master-error}. Since \(H\leq p^{1/2}\),
\[
\frac{p}{MH^{3/2}}\geq\frac1M,
\]
and hence \eqref{eq:relative-master} gives
\[
\frac{
\left|T_5(\lambda)-H^5M^5/p\right|
}{
H^5M^5/p
}
\ll_s
\frac{
p^{7/6}(H+\mu)^2
}{
H^{9/4}M^{13/6}
}
p^{o(1)}.
\]
The twelfth power of the factor preceding \(p^{o(1)}\) is
\[
\frac{
p^{14}(H+\mu)^{24}
}{
H^{27}M^{26}
}.
\]
The hypothesis of Theorem~\ref{thm:five-main} therefore implies
\[
\frac{
\left|T_5(\lambda)-H^5M^5/p\right|
}{
H^5M^5/p
}
\ll_s
p^{-\varepsilon/12+o(1)},
\]
which proves the theorem.
\end{proof}

\begin{proof}[Proof of Corollary~\ref{cor:five-balanced}]
First suppose that \(N\leq p^{1/2}\). Then \(\mu=N\), and the proof of
Theorem~\ref{thm:five-main} gives
\[
\frac{
\left|T_5(\lambda)-N^{10}/p\right|
}{
N^{10}/p
}
\ll_s
p^{7/6+o(1)}N^{-29/12}.
\]
The hypothesis \(N>p^{14/29+\varepsilon}\) therefore gives a fixed
power saving.

Suppose now that \(N\geq p^{1/2}\). From
\eqref{eq:relative-master}, with \(r=5\) and \(\ell=3\),
\[
\frac{
\left|T_5(\lambda)-N^{10}/p\right|
}{
N^{10}/p
}
\ll_s
pN^{-2}
\left(
\frac{p}{N^{5/2}}+\frac1N
\right)^{1/6}
p^{o(1)}.
\]
Using \((u+v)^{1/6}\ll u^{1/6}+v^{1/6}\), the right-hand side is
\[
\ll_s
\left(
p^{7/6}N^{-29/12}
+
pN^{-13/6}
\right)p^{o(1)}.
\]
For \(N\geq p^{1/2}\), these terms are respectively at most
\(p^{-1/24}\) and \(p^{-1/12}\). This proves the result.
\end{proof}

\section{Optimization in the balanced case}\label{sec:optimization}

We now set \(H=M=N\) and optimize the parameter \(\ell\).

Suppose first that \(N\leq p^{1/2}\). Since \(\mu=N\),
Corollary~\ref{cor:r-unbalanced} shows that it is enough to have
\[
N^{
(2\ell+r-2)/(\ell+1)
+(r-4)/(2\ell)
}
>
p^{1+(r-4)/(2\ell)+\varepsilon}.
\]
Thus the critical exponent is
\begin{equation}\label{eq:alpha}
\alpha_{r,\ell}
=
\frac{
(\ell+1)(2\ell+r-4)
}{
4\ell^2+(3r-8)\ell+r-4
}.
\end{equation}
More precisely, if \(N>p^{\alpha_{r,\ell}+\varepsilon}\), then the
left-hand side of the preceding sufficient condition exceeds its
right-hand side by a fixed power of \(p\), after replacing the
\(\varepsilon\) occurring in Corollary~\ref{cor:r-unbalanced} by a
suitable positive multiple depending only on \(r\) and \(\ell\).
Direct calculation gives
\[
\alpha_{r,\ell+1}-\alpha_{r,\ell}
=
-\frac{
2(r-4)(r-\ell^2+\ell-2)
}{
\bigl(4\ell^2+(3r-8)\ell+r-4\bigr)
\bigl(4\ell^2+3r\ell+4r-8\bigr)
}.
\]
Thus \(\alpha_{r,\ell}\) decreases while
\(r>\ell^2-\ell+2\) and increases once the reverse strict inequality
holds. Therefore an optimal integer parameter is
\[
\ell_r
=
\left\lceil
\frac{1+\sqrt{4r-7}}2
\right\rceil.
\]
If \(r=\ell_r^2-\ell_r+2\), then \(\ell_r\) and \(\ell_r+1\) are
tied.

It remains only to remove the temporary restriction
\(N\leq p^{1/2}\). For \(N\geq p^{1/2}\), put
\[
q=\frac{r-4}{2\ell_r},
\qquad
\theta=\frac{2\ell_r}{\ell_r+1}.
\]
Since \(r\geq5\), one has \(\ell_r\geq3\), hence \(\theta>1\).
From \eqref{eq:relative-master},
\[
\frac{
\left|T_r(\lambda)-N^{2r}/p\right|
}{
N^{2r}/p
}
\ll_{s,r}
pN^{-2}
\left(
\frac{p}{N^{1+\theta}}+\frac1N
\right)^q
p^{o(1)}.
\]
Here \(pN^{-2}\leq1\), while
\[
\frac{p}{N^{1+\theta}}
\leq
p^{-(\theta-1)/2},
\qquad
\frac1N\leq p^{-1/2}.
\]
Hence the relative error is \(O_{s,r}(p^{-c_r+o(1)})\) for some
\(c_r>0\), proving the asymptotic formula throughout
\(N\geq p^{1/2}\).

Finally, from \eqref{eq:alpha},
\[
\alpha_{r,\ell}-\frac13
=
\frac{
2(\ell^2+\ell+r-4)
}{
3\bigl(4\ell^2+(3r-8)\ell+r-4\bigr)
}.
\]
Since \(\ell_r=\sqrt r+O(1)\), the numerator is
\(4r+O(\sqrt r)\) and the denominator is
\(9r^{3/2}+O(r)\). Therefore
\[
\alpha_r
=
\frac13+\frac{4}{9\sqrt r}+O(r^{-1}),
\]
which completes the proof of Theorem~\ref{thm:r-balanced}.

\end{document}